\documentclass[twocolumn,leqno]{article}
\usepackage[alphabetic,msc-links]{amsrefs}

\usepackage[affil-it]{authblk}
\usepackage{graphicx}
\usepackage{color,bbm,subfig,caption,bm,fontenc,currvita}
\usepackage{amsmath,amsthm}
\usepackage{amsfonts}
\usepackage{amssymb,color,makeidx,enumerate,bbm}

\makeatletter
\renewcommand{\pod}[1]{\allowbreak\mathchoice
  {\if@display \mkern 18mu\else \mkern 8mu\fi (#1)}
  {\if@display \mkern 18mu\else \mkern 8mu\fi (#1)}
  {\mkern4mu(#1)}
  {\mkern4mu(#1)}
}
\newtheorem{theorem}{Theorem}
\newtheorem{ther}{Theorem}[section]

\newtheorem{example}[theorem]{Example}

\newtheorem{proposition}[theorem]{Proposition}

\newtheorem{definition}[theorem]{Definition}

\newtheorem{lemma}[theorem]{Lemma}

\newtheorem{corollary}[theorem]{Corollary}

\DeclareMathOperator{\spn}{span}

\newcommand{\koop}{\mathcal{K}}
\newcommand{\mathbm}{\bm}

\def\x{{\bf x}}
\def\n{{\bf n}}
\def\y{{\bf y}}
\def\z{{\bf z}}
\def\s{{\bf s}}

\def\w{{\bf w}}

\def\k{{\bf k}}
\def\m{{\bf m}}

\def\cB{{\cal{B}}}
\def\cD{{\cal{D}}}
\def\cO{{\cal{O}}}

\newcommand{\cH}{\mathcal{H}}
\newcommand{\caA}{\mathcal{A}}

\newcommand{\cl}{\mathrm{cl}}

\def\A{{\bf A}}
\def\F{{\bf F}}
\def\G{{\bf G}}
\def\H{{\bf H}}

\def\X{{\bf X}}
\def\v{{\bf v}}

\def\u{{\bf u}}
\def\f{{\bf f}}
\def\g{{\bf g}}
\def\h{{\bf h}}
\def\c{{\bf c}}

\def\bE{\mathbb{E}}
\def\bZ{\mathbb{Z}}

\def\bN{\mathbb{N}}
\def\bC{\mathbb{C}}
\def\C{\mathbb{C}}
\def\1{\mathbbm{1}}

\def\bR{\mathbb{R}}
\def\R{\mathbb{R}}
\def\bea{\begin{eqnarray}}
\def\eea{\end{eqnarray}}
\def\be{\begin{equation}}
\def\ee{\end{equation}}
\def\eps{\epsilon}

\def\rar{\rightarrow}
\usepackage{layout}

\title{Koopman Operator, Geometry, and Learning}
\author{
 Igor Mezi\'c\footnote{
    Department of Mechanical Engineering and Department of Mathematics, University of California, Santa Barbara.
    }
  }

\begin{document}
\maketitle
\begin{abstract}

We provide a framework for learning of dynamical systems rooted in the concept of representations and Koopman operators. The interplay between the two leads to the full description of systems that can be represented linearly in a finite dimension, based on the properties of the Koopman operator spectrum. The geometry of state space is connected to the notion of representation, both in the linear case - where it is related to joint level sets of eigenfunctions - and in the nonlinear representation case. As shown here, even nonlinear finite-dimensional representations can be learned using the Koopman operator framework, leading to a new class of representation eigenproblems. The connection to learning using neural networks is given. An extension of the Koopman operator theory to ``static" maps between different spaces is provided. The effect of the Koopman operator spectrum on Mori-Zwanzig type representations is discussed.
\end{abstract}
\section*{Introduction}

The end of the 20th and the beginning of the 21st century has seen a revolutionary 
increase in the availability of data. Indeed, we are in the middle of the {\it sensing} revolution,
where sensing is used in the broadest meaning of data acquisition. Most of this data goes unprocessed,
unanalyzed, and consequently, unused. This causes missed opportunities, in 
domains of vast societal importance - health, commerce, technology, security, just to mention some.

A variety of mathematical methods have emerged out of that need. Perhaps the most popular, methodology of Deep Neural Networks has as the underlying learning elements ``neuron functions", that
are modeled after biological neurons. The algorithms based on deep learning have
achieved substantial success in image recognition, speech recognition and natural language processing, deploying the ``supervised" machine learning philosophy.
Convolutional neural networks \cite{lecun2018power} provided a superstructure to the deep neural network architecture   that resembles the organization of the animal visual cortex.
This led to an enormous success in image recognition, and even in realistic image generation, via Generative Adversarial Networks (GAN's) \cite{goodfellow2014generative}. All of these examples are essentially static pattern recognition tasks. Deep Learning methodologies are less successful in dynamically changing contexts, present, for example, in autonomous driving.
It is interesting that, in contrast to biological modules responsible for vision,  even the basic issue of finding specific brain structures that are responsible for perception of time, and thus understanding of dynamics, are still being investigated \cite{roseboom2019activity}.

Koopman operator theory has recently emerged as the main candidate for machine learning of dynamical processes \cite{Mezic:2005,Budisicetal:2012,parmar2020survey}. Here, we briefly describe its history, emerging from efforts to extend the methodology used in quantum mechanics,  and describe the current focus, setting it within the new concept of dynamic process {\it representation}, and connecting along the way to the geometric dynamical systems theory methods that enable data-driven discovery of essential elements of the theory, for example stable and unstable manifolds. What emerges is a powerful framework for unsupervised learning from small amounts of data, enabling self-supervised learning \cite{lecun2018power} that is much more in line with the theory of human learning than the machine learning methods of the second wave \cite{prabhakar2017powerful}.

\section*{History}
\noindent Driven by success of the operator-based framework in quantum theory, Bernard Koopman proposed in his 1931 paper \cite{Koopman:1931} to treat classical mechanics in a similar way, using the spectral properties of the composition operator associated with dynamical system evolution. Koopman  extended this study in a joint work with von Neumann in \cite{koopman1932dynamical} \footnote{It is in this paper,  in which properties of continuous spectrum of the Koopman operator were investigated, that Koopman and von Neumann realized chaotic dynamics is possible for deterministic systems, stating ``Theorem I may also be expressed by saying that the states of motion
corresponding to any set $M$ of $\Omega$ become more and more spread out into
an amorphous everywhere dense chaos. Periodic orbits, and such like,
appear only as very special possibilities of negligible probability." This was 30 years before Lorenz and Ueda's contributions that started the chaotic dynamics revolution.}. Those works, restricted to Hamiltonian dynamical systems, did not attract much attention originally, as evidenced by the fact that between 1931 and 1990, the Koopman paper \cite{Koopman:1931} was cited 100 times, according to Google Scholar. This can be attributed largely to the major success of the geometric picture of dynamical systems theory in its state-space realization advocated by Poincar\'e. In fact, with Lorenz's discovery of a strange attractor in 1963, the dynamical systems community turned to studying dissipative systems and much progress has been made since. Within the current research in dynamical systems, some of the crucial roadblocks are associated with high-dimensionality of the problems and necessity of understanding behavior globally (away from the attractors) in  the state space. However, the weaknesses of the geometric approach are related exactly to its locality - as it often relies on perturbative expansions around a known geometrical object - and low-dimensionality, as it is hard to make progress in higher dimensional systems using geometry tools.

Out of today's 1000+ citations of Koopman's original work, \cite{Koopman:1931}, about 80\% come from the last 20 years. 
It was only in the 1990's  and 2000's that potential for wider applications of the Koopman operator-theoretic approach has been realized \cite{LasotaandMackey:1994,Mezic:1994,MezicandBanaszuk:2000,MezicandBanaszuk:2004,Mezic:2005,Rowleyetal:2009}. In the past decade the trend of applications of this approach has continued, as summarized in \cite{Budisicetal:2012,parmar2020survey}. This is partially due to the fact that strong connections have been made between the spectral properties of the Koopman operator for dissipative systems and the geometry of the state space. In fact, the hallmark of the work on the operator-theoretic approach in the last two decades is the linkage between geometrical properties of dynamical systems - whose study has been advocated and strongly developed by Poincar\'e and followers - with the geometrical properties of the level sets of Koopman eigenfunctions \cite{Mezic:1994,MauroyandMezic:2012,Mauroyetal:2013}. The operator-theoretic approach has been shown capable of detecting objects of key importance in geometric study, such as invariant sets, but doing so globally, as opposed to locally as in the geometric approach. It also provides an opportunity for study of high-dimensional evolution equations in terms of dynamical systems concepts \cite{Mezic:2005,Rowleyetal:2009} via a spectral decomposition, and links with associated numerical methods for such evolution equations \cite{Schmid:2010,Rowleyetal:2009}.

Even the early work in \cite{Mezic:1994,MezicandBanaszuk:2000} and its continuation in \cite{MezicandBanaszuk:2004,Mezic:2005,Rowleyetal:2009} already led to realization that spectral properties, and thus geometrical properties can be {\it learned} from data, thus initiating a strong connection that is forming today between machine learning and dynamical systems communities \cite{li2017extended,yeung2019learning,lusch2018deep,takeishi2017learning}. The key notion driving these developments is that of representation of a -possibly nonlinear - dynamical system as a linear operator on a typically infinite-dimensional space of functions. This then leads to search of linear, finite-dimensional invariant subspaces. In this paper we formalize the concept of {\it dynamical system representation}  enabling study of finite dimensional linear and nonlinear representations, learning, and the geometry of state space partitions. 

\section*{State Space vs. Observables Space}
It is customary, since Poincar\'e, to start the discussion of mathematics of dynamical systems with the notion of the state space.
However, to set the operator theoretic approach properly, it is useful to start with just the primitive notion of a set $M$ of states of a given system. Elements $\m\in M$ are abstract to start with, and the dynamics is the rule that assigns $\m^t\in M$ to $\m$ for any element $t$ of the time set $\mathcal{T}$. The time set can be $\bR$, $\bZ$, but more complicated cases such as $\mathcal{T}=\bZ^2$ can be considered as well.
As we are interested in framing the process of learning and modeling dynamics from data in the composition operator framework, we begin by describing the basic notions of representation of dynamics using functions.
\subsection*{Discrete Dynamical Systems}
The set $\cO$ of all complex functions $f:M\rar \bC$ is called the space of observables. It is  a linear vector space over the field of complex numbers.
A discrete deterministic dynamical system $T$ on $M$ is a map $M\rar M$, and the time set is $\mathcal{T}=\bZ$. For $\m\in M$  the iteration of  the map $\m$ is defined by $\m'=T\m$. 
Any such map defines an operator $U:\cO\rar\cO$ by
\be
Uf(\m)=f\circ T(\m)=f(T\m).
\ee
This is the composition operator associated with $T$, originally defined by Bernard Koopman \cite{Koopman:1931} in the context of measure-preserving transformations. In mathematics literature, it is often called the composition operator \cite{singh1993composition,cowen2019composition}. It is linear, as composition distributes over addition:
\bea
U(c_1f_1+c_2f_2)(\m)&=&(c_1f_1+c_2f_2)(T\m)\nonumber\\
&=&c_1f_1(T\m)+c_2f_2(T\m)\nonumber\\
&=&c_1Uf_1(\m)+c_2Uf_2(\m).
\eea
Ultimately, data is about numbers. We can understand a lot abut the map $T$ on $M$ by collecting data on observables. 
To formalize this, we need the notion of {\it representation}.
\begin{definition}
A finite-dimensional representation $(\f,\F)$ of $T$ in $\cO$ is a set of functions $\f=(f_1,...,f_n)$ and a mapping $\F$ such that
\be
U\f(\m)=\f(T\m)=\F(\f(\m)),
\label{rep}
\ee
where $\F:\bC^n\rar\bC^n$ and $n$ is the dimension of the representation.
If $\f:M\rar \bR^n$ is a real set of functions, then the representation is real. 
\end{definition} 
The image of $M$ in $\bC^n$ under $\f$ - the space $\f(M)$ - is called the {\it state space}. The simplest examples
of state spaces are Euclidean spaces of $n$-tuples of real numbers $\bR^n$. Consider $\f=(f_1,...,f_n)$ such that $f_j(\m)=m_j, \m=(m_1,...,m_n)\in\bR^n$.
Any mapping $\m'=T\m$ on $\bR^n$ has a real representation $\f'=\F(\f)$, where $\F_j(\f)=f_j(T\m)$.
The representation (\ref{rep}) is called {\it linear} provided $\F:\bC^n\rar\bC^n$ is a linear mapping.
Finite-dimensional representations are key to learning dynamical systems from data.\begin{example}\label{ex:torus}
Let $M=\mathbb{T}^2$, the two-dimensional torus, and $T:\mathbb{T}^2\rar\mathbb{T}^2$ the mapping that translates points on the torus by angle $\omega_1$ 
in the direction of rotation around the symmetry axis, and $\omega_2$ in the direction of the cross-sectional circle. Consider  the representation 
\bea
\f&=&(z_1,z_2)^T,\nonumber \\
\F(z_1,z_2)&=&(e^{i\omega_1}z_1,e^{i\omega_2}z_2),
\eea
 where $z_1=e^{i\theta_1}$, and $z_2=e^{i\theta_2}$, $\theta_1$ being the angle along the 
rotational symmetry, and $\theta_2$ angle along the cross sectional circle. We have
\be
\f'=A\f,
\ee
where $A$ is a diagonal matrix with $(e^{i\omega_1},e^{i\omega_2})$ as diagonal elements.
Thus, $(\f,\F)$ is a complex, linear representation of $T$.
\end{example}
Mathematically, one of the key questions in this context is whether a finite-dimensional representation exists. Namely, a set of functions $\f:M\rar \bC$ does not necessarily satisfy $\f\circ T=\F(\f)$. If one considers trajectories $\{\m_j,j\in \bZ \}$ of $T$, then it is easy to see that there isn't necessarily an $\F:\C^{\{J,-\infty\}}\rar \C,\{J,-\infty\}=\{J,J-1,...,J-N,...\}$ such that
\be
\f(\m_{J+1})=\f\circ T(\m_J)=\F(\f(\{\m_j,j\leq J\})),
\ee
 i.e. the next value of $\f$ can not always be be obtained uniquely even if we know the whole history of the evolution of $\f$ on the trajectory. The representation relationship $\f\circ T=\F(\f)$ requires that the next value of $\f$ is uniquely determined by the current value. This is the Markov property.
 If a representation does not satisfy Markov property, but its dynamics depends only on a finite number of previous trajectory points, i.e.
\bea
\f(\m_{J+1})&=&\f\circ T(\m_J)\nonumber \\
&=&\F(\f(\{\m_j,J-N\leq j\leq J\}))\nonumber \\
&=&\F(\f(\m_J),\f(\m_{J-1}),...,\f(\m_{J-N})) \nonumber \\
&=&\F(\f(T^{N}\m_{J-N}),\f(T^{N-1}\m_{J-N}),\nonumber \\
& &...,\f(\m_{J-N}))
\eea
then the so-called time-delay embedding \cite{Takens,Aayels} can be used to make it Markovian: let 
\be
\tilde  \f(\m)=(\f,\f\circ T,...\f\circ T^N)(\m)=(\f,\f_1...\f_N)(\m)
\ee
Then
\bea
\tilde  \f'&=&(\f\circ T(\m),\f\circ T^2(\m),...\f\circ T^{N+1}(\m))\nonumber \\
&=&(\f_1(\m),\f_2(\m),...,\f_N(\m),\f\circ T^{N+1}(\m))\nonumber \\
&=&(\f_1(\m),\f_2(\m),...,\f_N(\m),\f(\m_{N+1}))\nonumber \\
&=&(\f_1(\m),\f_2(\m),...,\f_N(\m),\nonumber \\
& &\F(\f(T^{N}\m),\f(T^{N-1}\m),...,\f(\m))\nonumber \\
&=&\tilde \F(\tilde \f).
\eea

Physically, there is the additional problem of whether experimental observations can provide all the information necessary to describe the finite-dimensional representation.
\subsection*{Representations and Conjugacies}
There are representations that are capable of separating points on $M$. We call these {\it faithful}:
\begin{definition} A representation $(\f,\F)$ is call faithful provided $\f:M\rar\f(M)$ is injective:
\bea
&(i)& \f(\m)= \f(\n) \Rightarrow \m=\n, \mbox{or equivalently}\nonumber \\
&(ii)& \m\neq\n \Rightarrow \f(\m)\neq \f(\n), \forall \m,\n\in M \nonumber
\eea
\end{definition}
In terms of representations, the Takens embedding theorem \cite{takens1981}  shows that a faithful representation of dynamics on  an $m$-dimensional Riemannian manifold can be obtained by using sufficiently large set of time-delayed observables for generic pairs of smooth real functions and dynamical systems $(f,T)$:
\begin{theorem} [Takens]Let $M$ be a compact Riemannian manifold of dimension $m$, $T:M\rar M$ a $C^r, r\geq 1$ diffeomorphism and $f:M\rar \bR$ a $C^r$ function. For generic $(f,T)$ the map $\tilde \f:M\rightarrow\mathbb{R}^{2m+1}$ given\ componentwise by 
\begin{equation*}
\tilde \f(x)=(f,f\circ T,f\circ T^{2},...,f\circ T^{2m}) 
\end{equation*}
is an embedding and $\tilde \f(M)$ is a compact submanifold of $\mathbb{R}
^{2m+1}$.  Thus, for generic $(f,T)$, $(\tilde\f,\F)$ is a faithful real representation of $T$.
\end{theorem}
Time-delayed observables have been used in approximations of Koopman operator theory since \cite{MezicandBanaszuk:2000,MezicandBanaszuk:2004}, followed by \cite{susuki2015prony,ArbabiandMezic:2017,brunton2017chaos}. Even more interestingly, and less obviously, Newtonian mechanics can be re-stated based on the idea of using time delays.
Indeed, Newton's second law for a single particle of constant mass moving on the real line states
\be
m\ddot x=F(x),
\ee
where $x$ is the standard euclidean coordinate function on $\R$. Since for smooth forces and configuration manifold $X$ the space of positions and momenta $M=X\times \bR$ is two dimensional, we can take one time delay to describe Newtonian motion.

A representation might provide redundant information: for example, it might contain two functions $f_j$ and $f_k$ such that $f_j=F(f_k)$ for some  function $F$. 
If it does not, it is called efficient:
\begin{definition} An $n$-dimensional faithful representation is called efficient provided 
there is no $H:\bC^{n-1}\rar \bC$ and $j\in\{1,...,n\}$ such that
\be
f_j=H(f_1,...,f_{j-1},f_{j+1},...,f_n).
\ee
\end{definition}
\begin{example}
In Example \ref{ex:torus} $(\f,\F)$ is a faithful, efficient representation of $T$.
\end{example}
It is clear that all efficient faithful representations have the same dimension. Thus, the dimension of the system can be defined as the dimension of 
its efficient representation\footnote{The underlying  space $M$ can have a fractal dimension -e.g. in the case of the Lorenz attractor - but the representation is integer-dimensional.}.
Additionally, different faithful representations  of the underlying mapping $T$ play nicely with each other as they are related by a conjugacy:
\begin{proposition} Let $(\f,\F),(\g,\G)$ be two different faithful $n$-dimensional representations. Then there is a bijection $\h:\g(M)\rar\f(M)$ such that 
\be
\f=\h(\g).
\ee
and $\h$ is a conjugacy of representations, i.e.
\be
\h(\G(\g))=\F(\h(\g))
\ee 
\label{prop:conj}
\end{proposition}
\begin{proof}
Since $\f$ and $\g$ are faithful, for every $\m\in M$ $\f(\m)$ and $\g(\m)$ are unique, and thus for any $\g\in \g(M)$ there is a unique $\f\in \f(M)$. The resulting mapping $\h: \g(M)\rar \f(M)$ is a bijection.
Further, we know 
\be
\h(\g\circ T))=\h(\G(\g))
\ee
Now,
\bea
\h(\g\circ T)&=&\f\circ T\nonumber \\
&=&\F(\f)\nonumber \\
&=&\F(\h(\g)),
\eea
and thus
\be
\h(\G(\g))=\F(\h(\g)).
\ee
\end{proof}
The concept of conjugacy has classically been used in dynamical systems for local linearization theorems \cite{zhang2017differentiability}. Since the Koopman operator description is global, extensions of the local theory are needed, as first provided in \cite{lanandmezic:2012b}. Since then, the concept of conjugacy has been utilized in Koopman operator theory of dynamical systems in a number of papers \cite{williams2015data,Mezic:2019,bollt2018matching,mohr2016koopman,eldering2018global}.
 
Faithful representations are capable of describing all of the dynamics of $T$. However, that dynamics is often high dimensional and has components that
are irrelevant for understanding of the problem at hand. In this context, the notion of the {\it reduced representation} is useful:
\begin{definition} A representation $(\f,\F)$ is called  reduced provided it is not faithful.
\end{definition}
Note that, by the definition of the concept of representation, even for reduced representations we have:
\be
 \f(\m)= \f(\n) \Rightarrow \f(T\m)=\f(T\n),  \forall \m,\n\in M. \nonumber
 \label{eq:valid}
\ee
since, if $ \f(\m)= \f(\n)$ then
\be
\f(T\m)=\F(\f(\m))=\F(\f(\n))=\f(T\n).
\ee
The concept of reduced representations is exemplified in the notion of factors in ergodic theory \cite{Petersen:1983,MezicandBanaszuk:2004}, for which we need to equip $M$ with a measure $\mu$. Let $(\f,\F)$ be a reduced representation of $T:M\rar M$, where the components of  $\f$ are measurable functions on $M$. Then we have
\begin{proposition}
The dynamical system $\f'=\F(\f)$ on $\f(M)$ equipped with the measure $\nu$ defined by $\nu(A)=\mu(\f^{-1}(A))$ is a factor of $T$.
\end{proposition}
\begin{proof}
We have
\be
\f(T\m)=\F(\f(\m)).
\ee
Since $\f$ is measurable, and $\nu(A)=\mu(\f^{-1}(A))$, $\F$ is a factor of $T$.
\end{proof}
In the context of factors, $\f$ is required to be measurable in contrast with the notion of  semi-conjugacy in topological dynamics,
where the representation is required to be continuous:
\begin{proposition} Let $\f:M\rar \f(M)$ be a continuous (proper) surjection i.e. there are two points in $M$ that map to a single point in $\f(M)$, and $(\f,\F)$ a  (non-faithful) representation. Then $\F$ is  semi-conjugate to $T$.
\end{proposition}
\begin{proof} We again have \be
\f(T\m)=\F(\f(\m)),
\ee
and thus $\F$ and $T$ are semi-conjugate.
\end{proof}
Both of these concepts - factors and semi-conjugacies - are key in model reduction of dynamical systems \cite{MezicandBanaszuk:2004,Mezic:2005,Mezic:2019}.
We discuss continuous time evolutions next
\subsection*{Representations of Continuous Time  Evolution}
In the case of continuous time $t\in \bR$, the evolution on $M$ consists of a group of transformations $T^t$, satisfying 
\be
T^{t+s}\m=T^tT^s\m, \forall t,s\in \bR
\ee
Any such evolution group defines an operator group $U^t:\cO\rar\cO$ by
\be
U^tf(\m)=f(T^t\m), f\in \cal{O}.
\ee

A representation $(\f,\F^t)$ of $T^t$ then consists  of a set of real or complex functions $\f$ and a group of transformations $\F^t$ that satisfy
\bea
\f(T^t\m)=\F^t(\f(\m))
\eea

For fixed $\tau$,  $U^\tau$ is a linear composition operator associated with $T^\tau$.
\begin{definition} A representation $(\f,\F^t)$ is called ordinary differential if it is finite and
\bea
\v(\m)&=&\frac{dU^t\f(\m)}{dt}=\lim_{t\rar 0}\frac{U^t\f(\m)-\f(\m)}{t}\nonumber \\
&=&\lim_{t\rar 0}\frac{\F^t(\f(\m))-\f(\m)}{t}
\eea
exists. In this case, the evolution is represented by a finite set of ordinary differential equations
\be
\dot{\f}=\v(\f).
\ee
\end{definition}
\begin{example}
Consider the set of all the states $\m$ of a mass-spring system, and the real representation $\f=(x(\m),p(\m))$, where $x$ is a numerical function that represents the deviation of the
mass position from the unstretched length of the spring and $p$ is the linear momentum, $p=\mathsf{m}v$ where $\mathsf{m}$ is the mass parameter, assumed constant, and  $v$ is an observable representing change of $x$ with time $t$. Derivatives with respect to time are labeled by $\dot{()}$. Then 
\bea
\dot x&=&p/\mathsf{m} \nonumber \\
\dot p&=&-kx \nonumber \\
\eea
is a two-dimensional, faithful, efficient, ordinary differential representation. Setting $\omega^2=k/m,z=x+ip$, we have a one-dimensional, faithful, efficient, complex representation
\be
\dot z=-i\omega z.
\ee
On the other hand, using energy $E=p^2/2m+kx^2/2$, we obtain a one-dimensional, reduced representation
\be
\dot E=0.
\ee
\end{example}
As the next example shows, simple representations can exist even for strange $M$:
\begin{example}[Lorenz representation] \label{exa:Lor}Let $M$ be the Lorenz attractor, which is a subset of $\bR^3$. Let $f(\m)=(x(\m),y(\m),z(\m))$
where $M$ is viewed as embedded in $\bR^3$ and $ (x(\m),y(\m),z(\m)), \  \m\in M$ are projections of $\m$ on $x,y,z$ axes. Then
\bea
\dot x&=&\sigma (y-x),\nonumber \\
\dot y&=&x(\rho -z)-y,\nonumber\\
\dot z&=&xy-\beta z.
\label{lcLorenz}
\eea
is a 3-dimensional efficient ordinary differential equation representation. Note that the underlying set $M$ is fractal, and yet is possesses a differential representation.
It is of interest to note that the ordinary differential equations (\ref{lcLorenz}) are valid off the set $M$ when it is viewed as embedded in $\bR^3$, 
but from the current point of view, the representation itself is only valid when restricted to $M$. 

The Lorenz representation is reduced and is in fact inexact as far as the full dynamic process it is supposed to represent is concerned: the dynamics it  models is that of a Boussinesq approximation of thermal convection dynamics, reduced by truncating Fourier series expansion of the solution.\end{example}

\subsection*{Infinite Dimensional Field Representations}
In the case where the representation of $T$ is not finite, and thus involves a field of observables, e.g  $\f\in L^2(P)$ for some continuous space $P$ (an example is $P=\bR$),
we speak of a field representation. 
\begin{example}
The scalar vorticity field 
\be
\omega(\x), \x=(x,y)\in A\subset \bR^2,
\ee of a two-dimensional, incompressible, inviscid fluid satisfies the equation
\be
\dot \omega={\cal N} \omega,
\ee
where $N:C^\infty(A)\rar C^\infty(A)$ is given by
\be
N\omega=\u\cdot\nabla \omega.
\ee
where
\be
\u(\x)=\frac{1}{2\pi} \iiint \frac{\omega\k\times(\x-\y)}{\|\x-\y\|^2} d\y,
\ee
and  $\k$ is the unit vector perpendicular to the plane of motion. This is the case of a partial differential representation.
\end{example}
\subsection*{Representations and data}
In general, it is not easy to find faithful representations, and their existence has to be validated experimentally. 
Namely, for a finite set of functions to comprise a representation, equation (\ref{eq:valid}) needs to hold for 
every $\m,\n\in M$. It is also clear that such validation is only possible for a finite set of points $\m,\n \in M$,
and thus there is always uncertainty present. It is easier to show that a representation is {\it not} faithful. For example, 
consider an object moving along a straight line, observed at time $0$.  It is impossible to predict where it will be at some time $\tau$ later
if we do not know its momentum $p$. Thus, the representation that includes only the observable $x$ (we could call it Galileo's)
- that measures the position along the straight line is not faithful. In contrast, the representation $(x,p)$ that contains both the position and momentum observables
is faithful (as long as the object does not interact with any others and thus proceeds moving with constant momentum).
This is Newton's representation, and it leads to ordinary differential equations
\bea
\dot x&=&p/m \nonumber \\
\dot p&=&0.
\eea
 
 Another aspect of representation that is important is its accuracy.
Namely, a set of functions $\f:M\rar \bR^m$ might be such that $|\f(T\m)-\F(\m)|\leq \epsilon$ for some $\epsilon$ that is small with respect to the 
average value of $\f$. This was the case, for example with Newton's representation of the motion under the law of gravity, where 
the orbit of planet Mercury, the closest to the sun, and thus experiencing the largest force of all planets, has a small but measurable deviation from the
motion predicted by the inverse-square law. This was rectified by Einstein's geodesic representation. But there are measurable deviations of motion of 
galaxies from the Einstein's representation, too. Like faithfulness, accuracy of representation is also measurable only up to experimental uncertainty.
 
\subsection*{Representations and  Geometry of State Space Partitions}
The key connection between the notion of representations and geometry is that of a partition.  The collection of disjoint sets $A_\alpha, \alpha \in \bC$ forms a partition $\zeta$ of $M$, iff
\be
\cup_{\alpha\in \bC}A_\alpha=M 
\ee
The partition $\zeta_f$ induced by an observable $f$ is defined by  
\be
\zeta_f=\{B^f_\alpha\in M|B^f_\alpha=\{\m\in M| f(\m)=\alpha, \alpha \in \bC\}\},
\ee
i.e. the sets $B^f_\alpha$ are level sets of $f$ on $M$, indexed over values $\alpha$ in the image of $f$ \cite{Rokhlin:1966}. The {\it product} $\vee$ of two partitions is defined by 
\be
\zeta\vee\xi=\{A\cap B|A\in \zeta,B\in \xi\}
\ee
The finest partition $\pi$ is the partition into individual elements of the set $M$.
\begin{proposition} A representation $(\f,\F)$, where  $\f=(f_1,...,f_n)$, is faithful if and only if the partition 
\be
\zeta_\f=\vee_{j=1,...,n}\zeta_{f_j}
\ee
is the finest partition $\pi$.
\end{proposition}
\begin{proof}
Assume $(\f,\F)$ is faithful and $\pi\neq \vee_{j=1,...,n}\zeta_{f_j}$. Then there are two points in $M$ that have the same values of $\f$ associated with them and we get a contradiction. Conversely, assume that $\pi=\vee_{j=1,...,n}\zeta_{f_j}$. Then, the partition is faithful as every point  of $M$ gets assigned a unique set of values of observables $f_1,...,f_n$
\end{proof}

\section*{Eigenfunctions and Linear Representations}
An eigenfunction $\phi$ of the composition operator associated with $T$ and  $T^t$ respecively satisfies
\bea
U\phi&=&\lambda \phi.\\
U^t\phi &=&e^{\lambda t} \phi
\eea
where $\lambda$ is the associated eigenvalue.
In discrete time, let $\F(\phi)=\lambda \phi.$ Thus, $(\phi,\F)$, is a (possibly reduced) representation of $T$. More generally, let $(\f,\F)$ be an $m$-dimensional representation of $T$ such that
$\F(\f)=A\f$ where $A$ is an $m\times m$ matrix. We denote such a linear representation by $(\f,A)$. Since we have
\be
U\f(\m)=\f\circ T(\m)=A\f(\m).
\ee
we call $A$ an eigenmatrix of $U$. Similar definition holds for the continuous time case, where we require
\be
U^t\f(\m)=\f\circ T^t(\m)=e^{At}\f(\m).
\ee
for $A$ to be an eigenmatrix of $U^t$. In the differentiable case, we get
\be
\frac{dU^t\f(\m)}{dt}|_{t=0}=A\f(\m).
\ee

Assume that $A$ has distinct eigenvalues. Let $\left<\cdot,\cdot\right>$ denote the standard complex inner product on $\bC^m$, defined by
$$
\langle\v,\w\rangle=\sum_i \v_i \cdot \w_i^c
$$
Then
\begin{proposition}
The eigenvalues of $A$, $\lambda_1,...,\lambda_m$ are eigenvalues of $U$, and the associated eigenfunctions of $U$ are given by
\be
\phi_j=\left<\f,\w_j\right>,
\ee
where $w_j$ is the eigenvector of $A^T$.
\end{proposition}
\begin{proof} In discrete time, we have
\bea
U\phi_j(\m)&=&\phi_j(T(\m))=\left<\f (T(\m)),\w_j\right>\nonumber \\
&=&\left<A\f(\m),\w_j\right>=\left<\f(\m),A^T\w_j\right>\nonumber \\
&=&\left<\f(\m),\lambda_j^c\w_j\right>=\lambda_j \left<\f(\m),\w_j\right>\nonumber \\
&=&\lambda_j \phi_j(\m). \nonumber
\eea
The proof for continuous time proceeds in a similar way.
\end{proof}
\begin{example}
The system $\dot x=f(x),x\in \bR^1, f\in C^1$  always has an eigenfunction satisfying $\dot \phi=\lambda\phi$
provided $\exp({\lambda\int_{x_0}^xdx/f(x)})\in C^1$ for some $\lambda<\infty$. The representation $(\phi,\G^t)$ is linear, where $\G^t(\phi)=e^{\lambda t}\phi$, and $\dot \phi=\lambda \phi$.
\end{example}
 Perhaps the most important and simplest representation, if it exists,
is the eigenfunction representation, given in discrete time by
\be
\f(T\m)=\Lambda(\f(\m)),
\label{rep1}
\ee
where $\Lambda$ is a diagonal $m\times m$ matrix and the diagonal elements $\Lambda_{jj}=\lambda_j,j=1,...,m$ are eigenvalues of $U,$
satisfying
\be
Uf_j=\lambda_jf_j,
\ee
where $f_j$ is  the eigenfunction of $U$ associated with the eigenvalue $\lambda_j$.

Let $A$ be an efficient faithful representation. Let $P$ be a diagonalizing matrix such that $P^{-1}AP=\Lambda$, where $\Lambda$ is a diagonal matrix, and $\bm{\phi}=(\phi_1,...,\phi_m)$. than $(\bm{\phi},\Lambda),$ is a linear representation conjugate to $(\f,A)$. This is of interest because it leads to the following corollarry:
\begin{corollary}
For any  linear diagonalizable finite-dimensional representation $(\f,\F)$ (respectively  $(\f,\F^t)$) of $T$ (respectively $T^t$), $\f$ is in the span of $n$ eigenfunctions $\bm{\phi}=(\phi_1,\phi_2,...,\phi_n)$ of $U$ (respectively $U^t$), where $n$
is the dimension of the representation.
\end{corollary}
\begin{proposition} 
 Consider a linear, diagonalizable, finite-dimensional representation $(\f,\F)$ (respectively  $(\f,\F^t)$) of $T$ (respectively $T^t$), and another 
  finite-dimensional representation $(\g,\G)$ (respectively  $(\g,\G^t)$) of $T$ (respectively $T^t$). Let $\h$ be a homeomorphism between them, $\f=\h(\g)$.
  Then $\g=\h^{-1}(B\bm{\phi})$ for some matrix $B$ of full rank.
\end{proposition}
\begin{proof}
Since $\f$ is in the span of $\bm{\phi}$, 
\be
\f=B\bm{\phi},
\ee
and $B$ is of full rank.
Since $\g=\h^{-1}\f=\h^{-1}(B\bm{\phi})$ we are done.
\end{proof}

\subsection*{Algebra of Eigenfunctions}
Eigenfunctions of $U$ form an algebra: let $\phi,\psi$ be eigenfunctions of $U$ with the associated eigenvalues $\lambda,\beta$.
Then $\phi\psi$ is also an eigenfunction, with eigenvalue $\lambda \beta$:
\be
U\phi\psi=(\phi\psi)\circ T=\phi \circ T\psi\circ T=\lambda \beta\phi\psi.
\label{alg}
\ee
 Thus, any  efficient representation $(\f,A)$ leads to many non-efficient representations that can be obtained by adding products of eigenfunctions into the set of representation functions.

\section*{Eigenfunctions, Geometry and Stability}
The above discussion leads to the conclusion that eigenfunctions of the Koopman family of operators $U^t$ are useful in the context in representations - not only is the
representation consisting of eigenfunctions linear, it is also fully decoupled, as each eigenfunction $\phi$ satisfies
\be
\dot \phi=\lambda \phi,
\ee
where  $\lambda$ is the associated eigenvalue. In the discrete-time case, eigenfunctions of $U$ similarly satisfy 
\be
\phi'=\phi\circ T=\lambda \phi.
\ee
The more general notion is that of a function $\phi^k$, such that for discrete maps $(U-\lambda I)^k\phi^k=0$, and for continuous time $(U^t-e^{\lambda t} I)^k\phi^k=0$. Such a function is called a generalized eigenfunction \cite{Mezic:2020}. Clearly, 
eigenfunctions satisfy such equations for $k=1$. Using generalized eigenfunctions, for a linear representation $(\f,A)$ we get a linear system
\be
\dot{\bm{\phi}}=J\bm{\phi}.
\ee
where $J$ is the Jordan canonical form matrix, and $\bm{\phi}=(\phi_1^1,...,\phi_1^{k_1},...,\phi_l^1,...,\phi_l^{k_l})$, $l$ being the number of
distinct eigenvalues, and $k_j$ the algebraic multiplicity of  eigenvalue $\lambda_j$.
\subsection*{Level Sets of Eigenfunctions and Invariant Partitions}
Level sets of eigenfunctions provide insight into geometry of the state space of a dynamical system. Consider an eigenfunction $\phi$ for $T:M\rar M$ at eigenvalue $1$.
It satisfies, in discrete and continuous time, respectively
\bea
\phi\circ T(\m)&=&\phi(\m) \nonumber \\
\dot \phi(\m)&=&0.
\eea
Therefore, $\phi$ is invariant under the dynamics of $T$, and its level sets, defined by $\phi=c,$ for some constant $c\in \bC$ are {\it invariant sets}.
Thus, learning of linear representations from data enables learning of invariant sets of the underlying system.  The partition $\zeta_\phi$ into the level sets of $\phi$ is an example of a {\it fixed partition}, since for any set $\A\in \zeta_\phi$, in discrete time $TA=A$. The finest such partition is the ergodic partition \cite{neumann1932,Rokhlin:1966,Mane:1987,Mezic:1994,MezicandWiggins:1999,susuki2018uniformly} that has interesting additional metric properties. The mapping $\phi:M\rar \C$ also defines a {\it fixed factor} of $T$, whose domain is $\phi(M)$ and its dynamics being trivial, and given by $\phi'=\phi$.
\begin{example} Consider the set of states $\m\in M=S^1\times \bR$ of a pendulum of mass $m$. Let $\theta\in [0,2\pi)$ and define  representation functions  $\theta(\m),\omega(\m)=\dot \theta(\m)\in \bR$.
We have the ordinary differential representation
\bea
\dot \theta&=&\omega  \nonumber \\
\dot \omega&=&-\frac{g}{l} \sin\theta \\
\eea
with $g$ the acceleration of gravity, $l$ the length of the pendulum and ${\cal I}m$ denoting the imaginary part of a complex number. Let 
\be
H=\omega^2/2+\frac{g}{l}\cos(\theta),
\ee
and $\dot H=0$. Thus, the Hamiltonian $H$ is an eigenfunction of the Koopman operator associated with pendulum dynamics.  Its level sets are invariant. The level sets of the Hamiltonian for $1$ degree of freedom systems form the ergodic partition, but this is not the case for $N$ degree of freedom Hamiltonian systems, since e.g. tori with irrational rotation dynamics can have half the dimension of the state space.
\end{example}
The eigenfunction $\phi_\omega$ of $T$ corresponding to an eigenvalue $\lambda=e^{i\omega}\neq 1$  on the unit circle yields level sets that form an {\it invariant partition}. Namely if $A\in \zeta_{\phi_\omega}$, then $TA=B \in \zeta_{\phi_\omega}$. If $n\omega=2\pi m$ where $m,n\in \bZ$ then $T^nA=A$ for every $A\in  \zeta_{\phi_\omega}$. In that case $\zeta_{\phi_\omega}$ is a {\it periodic partition}. The same analysis holds for continuous-time systems in the case when the eigenvalue is $\lambda=i\omega$, on the imaginary axis. For limit cycling systems, with limit cycling frequency $\omega$, there exists an eigenfunction $\phi_\omega$, the level sets of which satisfy
\be
\dot \phi_\omega=i\omega \phi_\omega.
\ee
Such level sets are isochrons \cite{winfree1974patterns,guckenheimer1975isochrons,MauroyandMezic:2012}. The notion of generalized isochrons in dynamical systems with a toroidal attractor with diophantine irrational rotation dynamics stems for further examination of partitions induced by imaginary axis eigenvalues \cite{MauroyandMezic:2012}.

More generally, consider an eigenfunction of $U$ that satisfies
\be
\phi'=\lambda \phi
\ee
with $|\lambda|<1$, or eigenfunction of $U^t$ that satisfies
\be
\dot \phi=\lambda  \phi,
\ee
for $\lambda \in \bC^-$, the left half plane (excluding the imaginary axis). Then necessarily, $\phi(t)\rar 0$ as $t\rar \infty$. The level sets of $\phi$ again form a {\it partition} of $M$ that is invariant. Namely the collection of disjoint sets $A_\alpha, \alpha \in \bC$ such that $A_\alpha=\{\m\in M| \phi(\m)=\alpha, \alpha \in \bC\}$ forms a partition of $M$, such that
\bea
\cup_{\alpha\in \bC}A_\alpha&=&M \nonumber \\
T^tA_\alpha&=&A_{ \alpha e^{ \lambda t}}
\eea
The second property indicates $A_\alpha$ is  an invariant partition under $T^t$. 
Level sets of Koopman eigenfunctions always provide us with invariant partitions of the state space.
\begin{proposition}
Let $\phi$ be a Koopman eigenfunction of a continuous time system $\dot \x=\F(\x)$ on $M\subset\R^n$ with the flow $T^t:M\rar M$, or of a map $T:M\rar M$. Then the level sets $A_\Phi$ of $\phi$
\be
A_\Phi=\{\x\in M| \phi(\x)=\Phi\},
\ee
where $\Phi\in \bC$ are elements of the invariant partition $\zeta_\phi=\{A_\Phi| \phi(\x\in \A_\Phi)=\Phi\}$.
\end{proposition}
\begin{proof}
We will do the proof for the continuous time case. By definition, $\phi$ satisfies
\be
\dot \phi=\lambda \phi
\ee
where $\lambda$ is the eigenvalue associated with $\phi$. Therefore
\be
\phi(T^t\x)=e^{\lambda t}\phi(\x),
\ee
and letting $\phi(\x)=\Phi$, we get $\phi(T^{t}\x)=e^{\lambda t}\Phi$. Since $e^{\lambda t}\Phi$ does not depend on $\x$, all the points in the set $A_\Phi$ get mapped into 
$A_{e^{\lambda t}\Phi}$ by $S^{t}, t\in \bR$ and thus $\zeta_\phi$ is an invariant partition. The proof in the discrete time case is similar.
\end{proof}
 Invariant partitions can be recurrent and non-recurrent:
\begin{definition} A recurrent invariant partition  \index{Recurrent invariant partition} \index{ Invariant partition! recurrent} of the state space is an invariant partition such that for any set A in it there is no $\tau$ such that $d(S^tA,A)\geq \eps>0,$ for some $\eps$ and all $t>\tau>0$. Here $d(A,B)$ is the Hausdorff distance of sets $A$ and $B$. An invariant partition that is not recurrent is called nonrecurrent.
\end{definition}
In other words, given an $\eps>0$, for any set $A$ in a recurrent partition, and for any $\tau,$ there is a time $t>\tau$ such that $d(S^tA,A)< \eps.$ Fixed and periodic partitions are clearly recurrent. So are partitions associated with an eigenvalue on the unit circle (discrete time) or imaginary axis (continuous time) where $\omega \neq 2\pi m/n$ for any $m,n \in \bZ$. 
\begin{example} \label{exa:lc}Consider the system   
\bea
\dot r&=& r(1-r^2)\nonumber \\
\dot \theta&=&\omega.
\eea
 The level sets of $r$ comprise an invariant non-recurrent partition for the system. 
However, $r$ is not an eigenfunction of the system. If we map every level set of $r$ into a single point we obtain the quotient space \index{Quotient space} $Q=[0,\infty)$. However the dynamics ``induced" on it by the mapping $r:M\rar R$ from the state space to $r$ is nonlinear. Using the eigenfunction of the system given by $\phi=(x^2-1)/x^2$ corresponding to the eigenvalue $-2$ (which is also the Floquet stability exponent for the limit cycle) we obtain linear dynamics $\dot \phi =-2\phi$ on $Q$. The invariant partition of $\bR^2$ into level sets of $\phi$ is nonrecurrent. The invariant partition into level sets of $\phi_\omega=e^{i2\pi \theta}$ corresponding to eigenvalue $\lambda=i\omega$ is recurrent (periodic).
\end{example}
The numerical study of invariant sets and invariant partitions using Koopman operator theory was initiated in \cite{Mezic:1994} and continued in \cite{MezicandWiggins:1999,levnajic2010ergodic,levnajic2015ergodic,das2017quantitative} for measure-preserving systems, where joint level sets of  time averages of continuous functions - that are elements of the eigenspace of $U$ at $1$ -  were used to compute invariant partitions. The figure \ref{Ergpart} from \cite{levnajic2015ergodic} shows numerical approximation using such time averages, which are part of Generalized Laplace Analysis, the larger computational framework in Koopman operator  theory (see below for more details). The mapping $T$ is  the standard map \cite{chirikov1979universal} on the 2-torus. We consider it in the form: 
\begin{equation}
\begin{array}{lllc}
  x' &=& x + y  + \eps \sin (2 \pi x)  \;\;\;  &[mod \; 1]  \\
  y' &=& y + \eps \sin (2 \pi x)  \;\;\;  &[mod \; 1]
\end{array}
\label{SM}
\end{equation}
where $(x,y) \in [0,1] \times [0,1] \; \equiv [0,1]^2$.  It is an area-preserving (symplectic) map which exhibits a variety of invariant sets, both regular, composed of  periodic or  quasi-periodic orbits, and chaotic zones that evolve in size and structure as the parameter $\eps$ is varied.

\begin{figure}[ht]
\centering
\includegraphics[height=2.048in,
width=3.118in ]{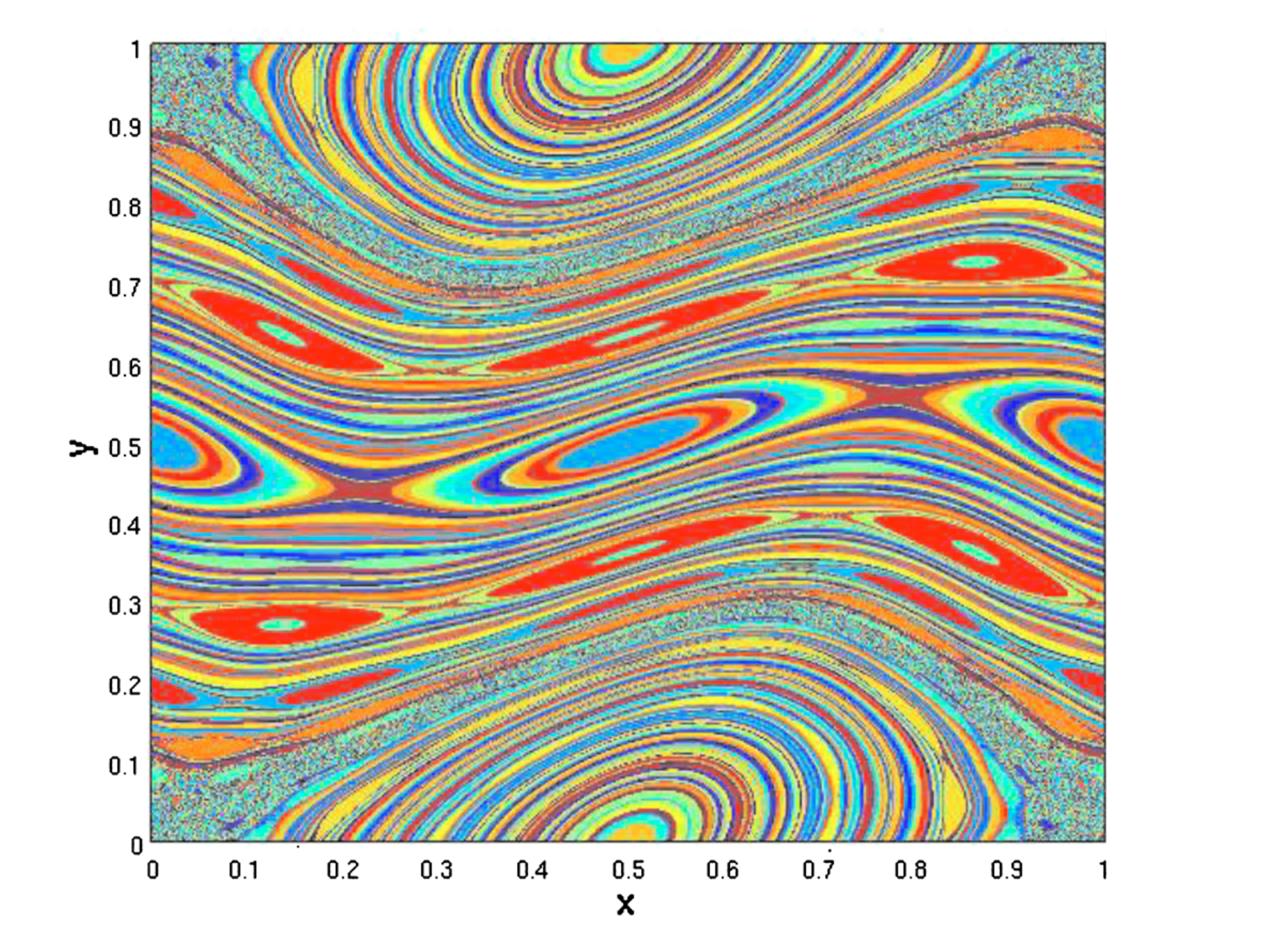} \caption{Two-function approximation of the ergodic partition of the standard map (\ref{SM}). From \cite{levnajic2015ergodic}.}\label{Ergpart}
\end{figure}
The above example indicates that joint level sets of several (or, in the ergodic partition case, countably infinite) eigenfunctions also provide invariant partitions. This concept can be used to compute stable, unstable, and center manifolds \cite{mezic2015applications,Mezic:2019}:
\begin{proposition}Let $T:M\rar M$ have a fixed point at $\m$. Let  $\lambda_1,...\lambda_u$ be positive real part eigenvalues, $\lambda_{u+1},...,\lambda_{u+c}$ be $0$ real part eigenvalues, and $\lambda_{u+c+1},...,\lambda_{s}$ be negative real part eigenvalues of a linear faithful efficient representation $(\x,A)$ with $\x(\m)=0$. Let \be\phi_1,...,\phi_{u+c+s},\ee be the (generalized) eigenfunctions of the associated Koopman operator.
Then the joint level set of (generalized) eigenfunctions
\be L_s=\{\x\in\R^n|\phi_1(\x)=0,...,\phi_{u+c}(\x)=0\},\ee is the stable subspace $E^s$, 
\bea L_c&=&\{\x\in\R^n|\phi_1(\x)=0,...,\phi_{u}(\x)=0,  \nonumber \\ &..., &\phi_{u+c+1}(\x)=0, \nonumber \\
&...,&\phi_{u+c+s}(\x)=0\},\eea is the center subspace $E^c$, and 
\be L_u=\{\x\in\R^n|\phi_{u+c+1}(\x)=0,...,\phi_{u+c+s}(\x)=0\},\ee the unstable subspace $E^u$. In turn,
$\f^{-1}(L_s),\f^{-1}(L_c),\f^{-1}(L_u)$ are the stable subset, the center subset and the unstable subset of $\m=\x^{-1}(0)$, the fixed point of $T$.
\end{proposition}
\begin{proof} The proof follows directly from Proposition 3.1 in \cite{Mezic:2019}.
\end{proof}
In the case when the fixed point is stable, the {\it inertial manifolds} \cite{constantin2012integral} can also be computed as joint zero level sets of a subset of Koopman operator eigenfunctions \cite{nakao2020spectral}.
\subsection*{Eigenfunctions and stability}
Provided we find eigenfunctions that compose a faithful representation, we can use them to examine stability properties of dynamical systems:
\begin{theorem} Let $(\bm{\phi}=(\phi_1,...,\phi_n),\Lambda)$, where $\Lambda$ is a diagonal matrix, be a faithful representation of $T^t$ such that $\lambda_j \in \bC^-,j=1,...,n$, and let $0\in \bm\phi(M)$. Then $\m|\bm{\phi}(\m)=0$ is a globally stable fixed point of $T^t$.
\end{theorem}
\begin{proof}
Clearly $\lim_{t\rar \infty} \bm{\phi}(\m_0)=0$ for any $\m_0\in M$. But $\bm\phi(M)$ contains $0$, and the representation is faithful. Thus assuming $\lim_{t\rar \infty} T^t \m_0\neq \m$ leads to contradiction.
\end{proof}
An analogous statement holds for discrete time $T$. The condition of faithfulness can be checked near the fixed point in the case of ordinary differential representations
\cite{mauroy2020koopman}.
\section*{Nonlinear Representations}
We have discussed linear representations that are based on finding eigenfunctions of the Koopman operator, and lead to
linear dynamics (reduction of the full Koopman operator) on a subspace of observables. Finite nonlinear representations
also lead to a reduction since the space of observables they operate on is the space of all observables that are constant on joint level sets of $\f$ 
- as in that case knowing $\f$ leads to knowing $\F(\f)$. The following simple lemma holds:
\begin{lemma}
The space $\cal{O}_\f$ of observables that are constant on joint level sets of $\f$ is a linear subspace of $\cal{O}$.
\end{lemma}
\begin{definition} A subspace of $\cal{O}$ is  said to be {\it generated} by a finite set of functions $\f$ if it is the subspace $\cal{O}_\f$ containing all observables $g(\f)\in \cal{O}$.
\end{definition}

\begin{corollary} A finite-dimensional invariant subspace of $U$ spanned by observables in a linear representation of dimension $n$ is a span of generalized eigenfunctions. Let $(\f,\F)$ be a finite-dimensional, nonlinear representation.
Then the subspace  $\cal{O}_\f$ generated by $\f$, is an  invariant subspace.
\end{corollary}
Thus, the search for finite-dimensional linear representations can be reduced to search for spans of generalized eigenfunctions. The search for nonlinear representations
is the search for invariant subspaces generated by finite sets of observables. It becomes clear that the eigenvalue-eigenfunction problem for the Koopman operator of finding $\phi$ and $\lambda$ such that
\be
U\phi=\phi\circ T=\lambda \phi,
\ee
is just a particular example of finding solutions of the {\it representation  eigenproblem} for a finite set of functions $\f$, and a map $\F$ that satisfy

\be
\boxed{U\f=\f\circ T=\F(\f).}
\label{eq:repeig}
\ee
in the particular case when $\F(\f)=A\f$ where $A$ is an $n\times n$ matrix, $A$ has been called an eigenmatrix \cite{LanandMezic:2013}. In the same vain,
we could call $\F$ an {\it eigenmap}.
\begin{example}
Recently, a proposition for learning nonlinear representations have been described in the SiNDY algorithm \cite{BruntonKutz}, starting with state observables $\f=\x=(x_1,...,x_n)$. $\F$ is expressed as a sum over a set of chosen functions $F_1,...,F_k$,
\be
\F=C F,
\ee
where $F=(F_1,...,F_k)^T$ and C is an $n\times k$ matrix.
The problem of finding $\F$ from data given as $(\x,\x\circ T,...,\x\circ T^m)$, in discrete time then reduces to the problem of finding $C$ such that
\be
X=\Theta C,
\ee
where
\be
X=\begin{bmatrix} 
\x\circ T\\
\vdots \\
\x \circ T^m \\
\end{bmatrix}
\ee
and
\be
\Theta=\begin{bmatrix} 
F_1(\x) & F_2(\x) & \cdots & F_k(\x)\\
F_1(T\x) & F_2(T\x) & \cdots & F_k(T\x)\\
\vdots \\
F_1(T^m\x) & F_2(T^m\x) & \cdots & F_k(T^m\x) \\
\end{bmatrix} 
\ee
The above setting can be formalized by the fact that any $\F$ can be expressed in a basis provided ${\cal O}$ is a separable Hilbert space.
\end{example}
Nonlinear representations can be reduced to linear representations provided a conjugacy exists.
\begin{proposition} Assume $(\f,\F)$ is an $n$-dimensional representation of $T$ and $\h$ is a conjugacy of $(\f,\F)$ to a linear representation $(\g,A)$. Then $\f$ is in a subspace of $\cO$ generated by the set of generalized eigenfunctions $\mathbm{\phi}=(\phi_1,...,\phi_n)$.
\label{prop:generate}
\end{proposition} 
\begin{proof}
Since $A$ is linear, there are eigenfunctions $\mathbm{\phi}=C\g$ where $C$ is an $n\times n$ invertible matrix such that
\be
\mathbm{\phi}'=\mathbm{\phi}\circ T=J\mathbm{\phi},
\ee
where $J$ is the Jordan normal form matrix for $A$.
 Since $\h$ is a conjugacy, $\h( \F(\f\circ T))=A\g(\h(\f\circ T))$ and
 $\f=\h^{-1}\g=\h^{-1}C^{-1}\mathbm{\phi}$, proving that $\f$ are in the invariant subspace generated by $\mathbm{\phi}$.
\end{proof}
The above result has a profound consequence on the issue of which systems can be rendered linear, that is in turn related to spectral properties of $U$.
\section*{The Spectral Triple}
So far, we have discussed the eigenvalues and eigenfunctions of the Koopman operator, and their connection to linear representations. Let $U$ act on some Banach space of observables. Then, the eigenvalues 
are part of the spectrum $\sigma(U)$, the complement of the residual set $\rho(U)$ defined as
\be
\rho(U)=\{\lambda\in \bC| (U-\lambda I)^{-1} \text{exists} \}.
\ee
The operator $(U-\lambda I)^{-1}$ is called the resolvent operator. The residual set, and thus the spectrum, are dependent on the functional space on which $U$ operates \cite{Mezic:2019}.
The operator $U$ is called scalar \cite{dundford:1954}  provided
\be
U=\int_{\sigma(U)}\beta dE,
\ee
where $E$ is a family of spectral projections forming resolution of the identity, and spectral provided
\be
U=S+N,
\ee
where $S$ is scalar and $N$ quasi-nilpotent. Examples of functional spaces in which Koopman operators are scalar and spectral are given in \cite{Mezic:2019}. Let $\f\in {\cal O}^n$ be a vector of observables. For the scalar Koopman operator $U$ the Koopman mode of $\f$ associated with an eigenvalue $\lambda$ is given by 
\be
\s_\lambda=\f_\lambda./\phi,
\ee
where $./$ is component-wise division, $\phi$ is the unit norm eigenfunction associated with $\lambda$, and
\be
\f_\lambda=\f-\int_{\sigma(U)/\lambda}\beta dE(\f).
\ee
We assume that the dynamical system $T$ has a Milnor attractor ${\cal A}$ \cite{milnor1985concept} such that for every continuous function $g$,
for almost every $\m:M\rar M$ with respect to an a-priori measure $\nu$ on $M$ (without loss of generality as we can replace $M$ with the basin of attraction of  ${\cal A}$)
the limit
    \be
   g^*(\m)=\lim_{n\rar \infty} \frac{1}{n}\sum_{i=0}^{n-1} U^i g(\m),
   \label{eqn:ta}
    \ee
    exists. This is the case e.g. for smooth systems on subsets of $\R^n$ with Sinai-Bowen-Ruelle measures, where $\nu$ is the Lebesgue measure \cite{hunt1998unique}.  Let $\g(\z,\m)$ be a field of observables indexed by the field variable $\z$ (e.g. $\g(\z,\m)$ could be temperature at spatial position $\z$ when  the system is at $\m\in M$). The {\it spectral expansion}  of the action of $U$ on $\g$ in $\cO=L^2(\mu)$ is given by \cite{Mezic:2005}
\begin{eqnarray*}\label{invmeasexp}
U\g(\z,\m) 
&=&\g^{\ast }(\z,\m)\\
&+&\sum_{j=1}^{k}\exp(i\omega_{j}t) \phi_{j}(\m)\s_j(\z) \\
&+& \int_{0}^{1 }\exp (i2\pi
\alpha t )dE(\alpha)(\g(\z,\m)),
\end{eqnarray*}
where $\g^{\ast }(\z,\m)$ is the time average (\ref{eqn:ta}), $\lambda_j=i\omega_j$ is an eigenvalue with the associated eigenfunction $\phi_j$, and  $\s_j(\z)$ is the $j$-th Koopman mode, i.e. the projection of $\g$ on the eigenspace of the eigenfunction $\phi_j$. The triple $(\lambda_j,\phi_j,\s_j)$ is called the spectral triple. From the previous discussion, any finite set of $\phi_j$'s provides for a (diagonal!) linear representation of $T$.

The term \be\int_{0}^{1 }\exp (i2\pi
\alpha t )dE(\alpha)(\g(\z,\m))\in\cO_{cont}\subset \cO\ee is projection of $\g$ on the continuous part of the spectrum, that is orthogonal to the  \be\cO_{eig}={\sf cl}({\spn\{\phi_j,j=1,...,\infty\}})\ee where ${\sf cl}$ stands for closure. Any finite-dimensional representation of $T$ in $\cO_{cont}$ must be nonlinear. 
\begin{corollary}[To proposition \ref{prop:generate}] A finite-dimensional representation $(\f,\F)$ is not conjugate to a linear representation
provided $\f \not\perp \cO_{cont}$.
\end{corollary}
\begin{proof} Assume the conjugacy to a linear representation exists. Then $U$ has point spectrum \cite{Mezic:2019}
and therefore $\f \perp \cO_{cont}$.
\end{proof}
\begin{example}
Consider the following coupling of the limit cycling system in example (\ref{exa:lc}) and the Lorenz system in example (\ref{exa:Lor}):
\bea
\dot r&=& (1+f(x,y,z))r(1-r^2)\nonumber \\
\dot \theta&=&\omega \nonumber \\
\dot x&=&\sigma (y-x),\\
\dot y&=&x(\rho -z)-y,\nonumber\\
\dot z&=&xy-\beta z. \nonumber\\
\label{lcLorenz}
\eea
where $f$ is some positive bounded real function $0<f<c<\infty$. The representation of this system is part linear and part nonlinear, where the complex function $\psi=e^{i\theta}$ constitutes the linear part 
\be
\dot \psi=i\omega \psi.
\ee
\end{example}
Similar results are available for spectral expansions of a large class of systems with Milnor attractors, see \cite{Mezic:2019}.

\subsection*{Types of Spectra}
There are two elements that determine the spectrum of a given dynamical system: the function space and the type of the attractor determined by its dynamics \cite{Mezic:2019}. Interestingly, these are commingled: a linear dynamical system in a complex plane can have a fixed point, where on a subset of linear observables  the spectrum is discrete, but in $L^2(\nu),$ where $\nu$ is Lebesgue, will have a very large spectrum, for example filling the entire unit disk of the complex plane \cite{ridge:1973}. The ``on-atrractor" space can always be chosen to be $L^2(\mu)$. But the transient dynamics requires spaces {\it adapted} to the dynamics, as described in \cite{Mezic:2019}. In such spaces, the on-attractor spectrum and the off-attractor spectrum combine by multiplication to provide the full spectrum of teh Koopman operator:
For a scalar Koopman operator of a dynamical system with a Milnor attractor with $\mu$ being a Borel measure, define the tensor product space \begin{equation}\cH=\cH_\caA\otimes \cH_\cB.
\end{equation} where $\cH_\caA=L^2(\mu)$ and Define $\cH_\cB=\tilde \cH_\cB\cup {\mathbf 1}$, where ${\mathbf 1}$ is the constant unit function on $\cD$ and $\tilde\cH_\cB\subset C(\cD)$  a Hilbert space of  functions $f:\cD\rightarrow \mathbb{C}$ that vanish on the attractor $\caA$.  Clearly,
$U^t=U^t|_{\cH_\caA}\otimes U^t|_{\cH_\cB}$ on $\cH$. Define  $P(a,b)=a\cdot b, \ a,b\in \bC$ to be the scalar product of $a$ and $b$, and 
\begin{equation}
P(A,B)=\cup_{a\in A,b\in B}P(a,b), \ \ A,B\subset \bC.
\end{equation}
We have the following \cite{Mezic:2019}:
\begin{theorem} \label{dissHilb}Consider the composition operator $U^t:\cH\rightarrow \cH$, and let $\sigma(U^t|_{\cH_\caA}),\sigma(\koop|_{\cH_\cB})$ be the spectra of its restrictions to $\cH_\caA$ and $\cH_\cB$ with the associated projection-valued spectral measures $P_\omega, \ \omega\in S^1$, and $P_z,\ z\in \mathbb{C}$. Then $\sigma(U^t)=\cl({P(\sigma(U^t|_{\cH_\caA}),\sigma(U^t|_{\cH_\cB})))}$ and 
\label{the:product}
\begin{equation}
U^t=\int_{\mathbb{C}} \int_{\mathbb{R}} e^{z t} e^{i2\pi\omega t} dP_{\omega}dP_{z}.
\label{spectral}
\end{equation}
\end{theorem}
\begin{example}
For a continuous-time, globally stable limit cycle system in $\bR^{n+1}$ with limit cycle frequency $\omega$, the on-attractor spectrum is $in\omega, n\in \bZ$. The off-attractor spectrum in Modulated Fock Space \cite{Mezic:2020} consists of positive integer combinations $\lambda_{\m,k}=\m\cdot\bm{\beta}, \m\in \bN^n$ of Floquet stability exponents $\bm{\beta}=(\beta_1,...,\beta_n)$. Thus, the full spectrum in $\cH$ is given by $in\omega+\m\cdot\bm{\beta}, n\in \bZ, \m\in \bN^n$.

Consider the three-dimensional, limit cycling system
\begin{eqnarray}
\dot x&=&y  \label{eq:1}\\
\dot y&=&x-x^3-cy  \label{eq:2}\\
\dot \theta&=&\omega. \label{eq:3}
\end{eqnarray}
The two fixed points of the equations (\ref{eq:1}-\ref{eq:2}) are $y=0,x=\pm1$. The linearization matrix at those is 
\be
A=\left[ \begin{array}{cc} 0 & 1 \\ -2 & -c \end{array}\right],
\ee
and thus the eigenvalues are determined by
\be
(-\lambda)(-c-\lambda)+2=\lambda^2+c\lambda+2=0,
\ee
leading to
\be
\lambda_{1,2}=\frac{-c\pm\sqrt{c^2-8}}{2}
\ee
For $c=\sqrt{7},\omega=1$, the eigenvalues  read $\lambda_{3,4}=-1.3228756\pm .5i$. Setting $\omega=1$, the other two principal eigenvalues are $\pm i$.
In figure \ref{EvaLat} (from \cite{Mezic:2019}) we show a subset of the eigenvalues of the Koopman operator on $L^2(S^1)\times {\cal A}$, where ${\cal A}$ is the space of analytic functions on the plane, in the basin of attraction of either of the limit cycles (since they are symmetric) of (\ref{eq:1}-\ref{eq:3}).
\begin{figure}[h!!] 
\centering \hspace{.5cm} \includegraphics[width=3in,
height=3in, clip=true, trim=0 100 0 100]{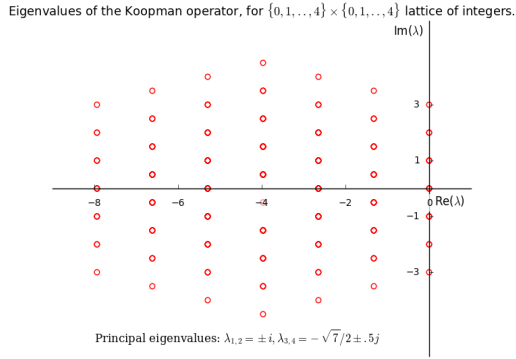}
\caption{ Eigenvalues of the Koopman operator, for $\{0,1,..,4\}\times \{0,1,..,4\}$ lattice of integers. From \cite{Mezic:2019}.}
\label{EvaLat}
\end{figure}
While the computations of point spectrum were already available in \cite{MezicandBanaszuk:2004} using GLA, and \cite{Rowleyetal:2009} using DMD, computations of continuous spectrum are more recent \cite{Kordaetal:2020,govindarajanetal:2019,giannakis2020delay}. They have been used to identify coherent pseudo-eigenfunctions in the (mixing!) Lorenz system. The contour plot of such eigenfunction is shown in figure \ref{Lor} from \cite{Kordaetal:2020}.
\begin{figure}[h!!] 
\centering \hspace{.5cm} \includegraphics[width=3in,
height=2in, clip=true, trim=0 100 0 100]{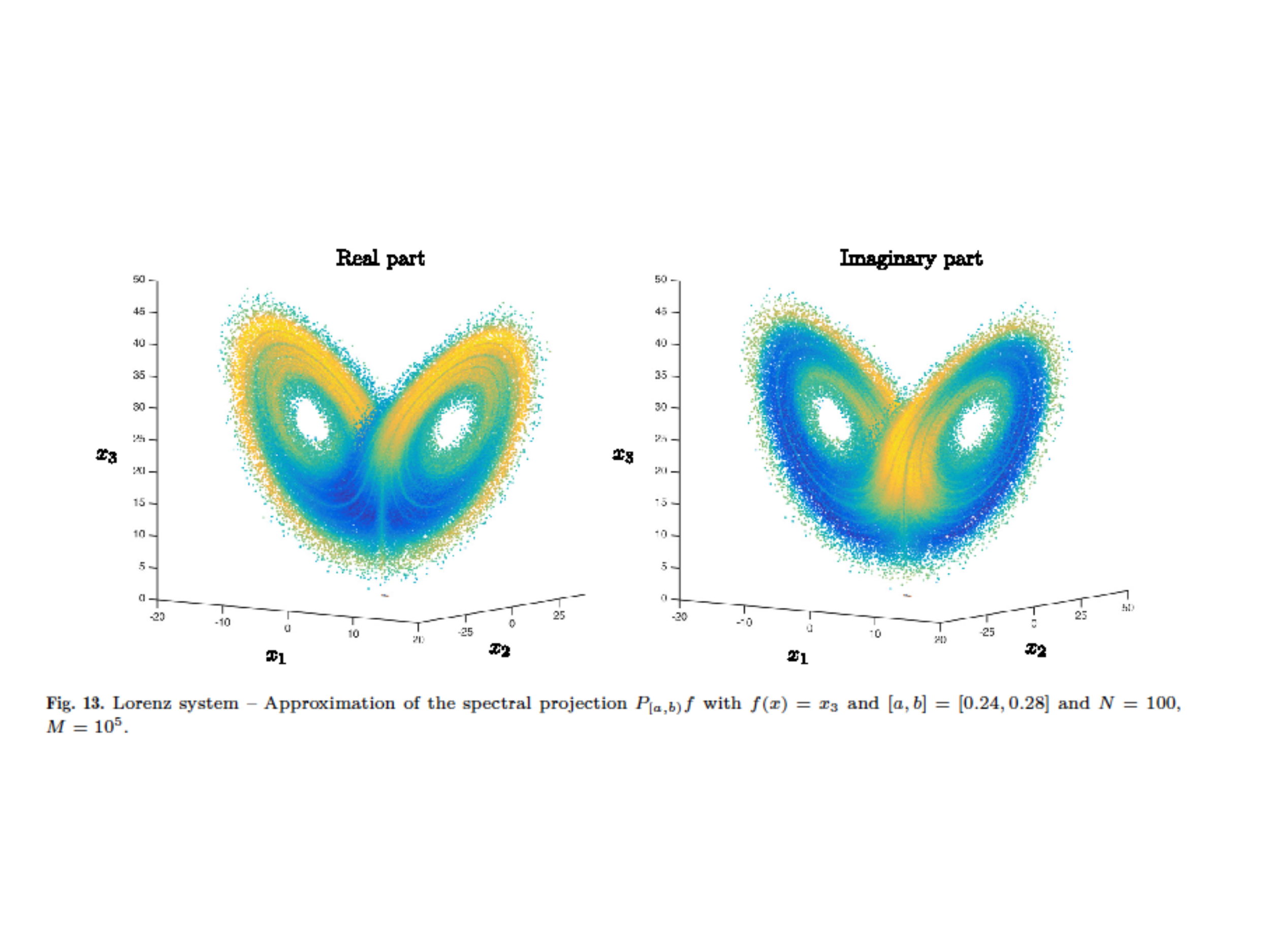}
\caption{Lorenz system: Approximation of the spectral projection $P[a,b)f$ with $f(x) =x_3$ and $[a, b] =[0.24, 0.28]$. The contours of the real and imaginary parts of the resulting pseudo-eigenfunction are indicated by color, showing substantial regularity. From \cite{Kordaetal:2020}. }
\label{Lor}
\end{figure}
Such observables have much longer prediction horizons than a typical observable on the Lorenz system.
\end{example}
\section*{Learning Dynamical Systems from Data}
Historically, the dynamical equations of motion, such as Newton's, Einstein's and Schr\"odinger's were obtained using depth of human intuition guided by small amounts of, or no data (Einstein's case). 
Classical automatized approaches to learning dynamical systems from data arose in control theory \cite{ljung1994modeling}. The goal was to connect system inputs $\u$ to system outputs $\y$ via analysis of a structured model connecting these. The most commonly used structure of the model is linear
\bea
\dot \x&=&A\x+B\u \nonumber \\
\y&=&C\x
\eea
where $\x$ is the state of the system, and $\y$ a linear  vector of observables.

 In
{\it static} machine learning problems  there are also ``inputs" and ``outputs", in the simplest case $\x\in \bR^n$  and $\y\in \bR^m$, although both input and output spaces can be more complicated, say manifolds. The key objective is to connect inputs and outputs by a map $f:\bR^n\rar \bR^m$ learned from a measured subset of input-output pairs 
\[
(\x_j,\y_j),j=1,...,N.
\]
Let $Im(f)$ be the image of the map $f$, and $Dom(f)$ its domain.
Provided $Im(f)\subset Dom(f)$, $f$ could be considered a dynamical system, since $f^k(\x)$ is well defined for any $\x\in Dom(f)$. In this case the data pairs $(\x_j,\y_j)$ can be obtained as successive points along the trajectory of $f$:
\be
\y_j=f(\x_j)= f^j (\x)
\ee
The learning problem, in both the cases of  static and dynamical systems $f$ is the same: given the data pairs, approximate $f$ for any input point. The dynamics does provide an advantage though, as data can be sampled along a trajectory advancing in time. Assume a discrete time dynamical system $T$ has an $n$-dimensional linear representation $(\f,A)$, such that
\be
\f'=A\f.
\ee
If we take samples of $\f$ along a trajectory $\f(\m),\f(T\m),...,\f(T^m\m)$, obtaining a sequence of {\it snapshots}
\be
\f_1=\f(\m),\f(T\m)=\f_2,...,\f(T^m\m)=\f_m,
\ee
we have
\be
\f(T^k)=A\f(T^{k-1}), k=1,....,m.
\ee
Thus, forming  data matrices
\be
X=
\begin{bmatrix} 
f_1(\m) & f_1 (T\m)& \cdots & f_1(T^{m-1}\m)\\
f_2(\m) & f_2 (T\m)& \cdots & f_2(T^{m-1}\m)\\
\vdots \\
f_n(\m) & f_n (T\m)& \cdots & f_n(T^{m-1}\m \\
\end{bmatrix} 
\ee
and
\be
X'=
\begin{bmatrix} 
f_1(T\m) & f_1 (T^2\m)& \cdots & f_1(T^m\m)\\
f_2(T\m) & f_2 (T^2\m)& \cdots & f_2(T^m\m)\\
\vdots \\
f_n(T\m) & f_n (T^2\m)& \cdots & f_n(T^m\m) \\
\end{bmatrix} 
\ee
Note that each row $j$ of data matrices $X,X'$ is an evaluations of the function $f_j$ on the trajectory of $T$ 
starting at $M$. Setting 
\be
X'=XC,
\ee
we see that $C$ is the companion matrix
\be
\tilde U=C=\left(
\begin{matrix}
0 &  0  & \ldots & 0 &c_{1}\\
1 &  0 & \ldots & 0& c_{2}\\
0 &  1 & \ldots & 0& c_{3}\\
\vdots & \vdots & \ddots &  \vdots & \vdots\\
0  &   0       &\ldots & 1 & c_{m} \\
 \end{matrix}\right)
 \label{eq:comp}
\ee
where $\c=(c_1,...,c_m)=\c_{m}$. 
The solution of this equation, provided $m=n$, and $X$ is non-singular, is
\be
C=X^{-1}X'.
\label{eqn:linsub}
\ee
The matrix $C$ would then be thought of as an approximation to the Koopman operator acting on the space $\cO_m=\bR^m$ of functions on the set of points $Tr=(\m,T\m,...,T^m\m)$ \cite{Mezic:2020}.
There are several caveats here. Typically we do not know in advance that a linear representation exists. Thus, we need to specify the dimension, $n$ and choose $f_1,...,f_n$. In the case when $\f$ do not span an $n$-dimensional invariant subspace of $U$, we can as well identify $C$ from (\ref{eqn:linsub}), but we can have 
\be
f_j(T^{m})\neq \sum_{k=1}^{m-1}c_k f_j(T^k\m).
\ee
In other words, it might be that the $m$-th element of the Krylov sequence $f_j(\m), f_j(T\m),...$ does not belong to the same subspace as the first $m-1$ elements. In addition, all the functions $f_k,k=1,...,n$ would have the same relationship between the last element of their own Krylov sequence and the prior elements. If the trajectory $Tr$ is periodic with period $m$, then $\c_m=(1,0,...,0)$ provides an exact reduction of the Koopman operator to $\cO_m$. Even if the trajectory of $T$ is dense on a subset, the approximation has good properties \cite{Mezic:2020}. In principle, for such trajectories the number of {\it snapshots} $\f_1,...,\f_m$ can be smaller and than the number of functions in each snapshot, and
\be
C=X^+X,
\ee
where 
\be
X^+=(X^\dagger X) X^\dagger
\ee
is the Moore-Penrose pseudoinverse of $X$, and $X^\dagger$ is the hermitian transpose (we allow for complex observables) of $X$.
 This was the reason behind the initial success utilizing the methodology of Dynamic Mode Decomposition \cite{Schmid:2010} to complex fluid flows in \cite{Rowleyetal:2009}, as the Koopman modes uncovered there were the consequence of the quasi-periodic nature of a portion of the underlying attractor \cite{Mezic:2005}.

A more structured approach to finding linear representations is that of finite-section methods, the first version of which was Extended Dynamic Mode Decomposition (EDMD) \cite{williamsetal:2015}. Consider the Koopman operator acting on an observable space ${\cO}$ of functions on the state space $M$, equipped with the complex inner product $\left<\cdot,\cdot\right>$,\footnote{Note hat we are using the complex inner product linear in the first argument here. The physics literature typically employs the so-called Dirac notation, where the inner product is linear in its second argument.} and let
$\f=\{f_j\},\ j\in \bN$ be an orthonormal basis on ${\cO}$, such that, for any function $f\in {\cO}$ we have
\be
f=\sum_{j\in \bN}c_jf_j.
\ee
Let
\be
u_{kj}=\left<Uf_j,f_k\right>.
\label{eq:inn}
\ee
Then,
\be
(Uf)_k=\left<Uf,f_k\right>=\sum_{j\in \bN}c_j\left<Uf_j,f_k\right>=\sum_{j\in \bN} u_{kj} c_j.
\ee

The basis functions do not necessarily need to be orthogonal. Consider the action of $U$ on an individual, basis function $f_j$:
\be
Uf_j=\sum_{k\in \bN}u_{kj} f_k,
\ee
where $u_{kj}$ are now just coefficients of $Uf_j$ in the basis.
We obtain
\be
Uf=\sum_{j\in \bN}c_jUf_j=\sum_{j\in \bN}c_j\sum_{k\in \bN}u_{kj} f_k=\sum_{k\in \bN}\left(\sum_{j\in \bN}u_{kj}c_j\right) f_k,
\ee
and we again have 
\be
(Uf)_k=\sum_{j\in \bN} u_{kj}c_j.
\ee
As in the previous section, associated with any linear subspace $\mathcal G$ of $\cO$, there is a projection onto
it, denoted $P = P^2$, that we can think 
of as projection ``along" the space $(I-P){\cO}$, since, for any $f\in {\cO}$ we have
\be P(I-P)f=(P-P^2)f=0,
\ee
and thus any element of $(I-P){\cO}$ has projection $0$. 
We denote by $\tilde U$ the infinite-dimensional matrix with elements $u_{kj},\ k,j \in \bN$. Thus, the finite-dimensional section of the matrix
\be
\tilde U_{n}=\begin{bmatrix}
		 u_{11} &  u_{12} & \cdots &   &  u_{1n}\\
		  u_{21}&  u_{22} &        &   &  u_{2n}\\
		&  &        &  & \\
		\vdots & & \ddots & &\vdots\\
		 u_{n1} & u_{n2} & \cdots & & u_{nn}
	\end{bmatrix},
	\label{comp}
\ee
is the so-called compression of $\tilde U$ that satisfies
\be
\tilde U_n=P_n\tilde U P_n,
\ee
where $P_n$ is the projection ``along" $
(I-P_n){\cO}$ to the span of the first $n$ basis functions, $\spn({f_1,...,f_n})$. To apply the finite-section methodology of approximation of the Koopman operator, we need to estimate coefficients $u_{kj}$ from data.

In the case of non-orthonormal basis, denote by $\hat f_k$ the dual basis vectors, such that
\be
\left<f_j,\hat f_k\right>=\delta_{jk},
\ee
where $\delta_{jj}=1$ for any $j$, and $ \delta_{jk}=0$ if $j\neq k$. 
For the infinite-dimensional Koopman matrix coefficients we get 
\be
u_{kj}=\left<Uf_j, \hat f_k\right>.
\ee
Let's consider the {\it finite} set of independent functions $\tilde \f=\{f_1,..,f_N\}$ and the associated dual set $\{\hat g_1,..,\hat g_N\}$ in the span  $\tilde {\cO}$ of $\tilde \f$, that satisfy
\be
\left<f_j,\hat g_k\right>=\delta_{jk}.
\ee
 Under ergodicity condition, in the case $\cO=L^2(\mu)$ it is possible to obtain $\tilde u_{kj}$ from data \cite{Mezic:2020}: 
\begin{ther}
\label{thernonor}
Let $\{f_1,..,f_N\}$ be a set of functions in $L^2(M,\mu)$ and let $T$ be ergodic on $M$ with respect to an invariant measure $\mu$. Let $\x_l,l\in \bN$ be a trajectory on $M$. Then, for almost any $\x_1\in M$
\bea
\label{eq:finsecentries}
\tilde u_{kj}&=&\lim_{m\rar\infty}\frac{1}{m}\sum_{l=1}^m f_j \circ T(\x_l) \hat g_k^c(\x_l)\nonumber \\
&=&\lim_{m\rar\infty}\frac{1}{m}\sum_{l=1}^m  f_j (\x_{l+1}) \hat g_k^c(\x_l),
\eea
where, for any finite $m$, $\hat g_k^c(\x_l),l=1,...,m$ are obtained as rows of the matrix $(F^{\dagger}F)^{-1}F^{\dagger}$, where
\be
F=\left[f_1(\X) \ f_2(\X)\ ... \ f_N(\X)\right],
\ee
$F^{\dagger}=(F^c)^{T}$ is the conjugate (Hermitian) transpose \index{Transpose! conjugate} \index{Conjugate transpose} \index{Hermitian transpose} of $F$
, and $f_j(\X)$ is the column vector $(f_j(\x_1) \ ... \ f_j(\x_m))^T$.
 \label{ther:finsecdata1}
\end{ther}
The above result is convenient as a single trajectory of an ergodic system can be used to estimate the inner product. But the formulation is restricted to measure-preserving systems. Alternatively, the above formula is valid  in any case where the function space is a Hilbert space, and inner product can be defined as a weighted sum over sample points $\x_l$,  
\be
\label{eq:finsecentries1}
\tilde u_{kj}=\lim_{m\rar\infty}\frac{1}{m}\sum_{l=1}^m  w(\x_l)f_j (\x_{l+1}) \hat g_k^c(\x_l).
\ee
\begin{proposition}
Let $(M,\mu)$ be a measure space and $T:M\rar M$. Let $\f=(f_{j_1},...,f_{j_n})$ be a subset of a basis in a Hilbert space of observables $\cO =L^2(\mu)$. Let $u_{j_kl}=0$ whenever $k\in{1,...,n},l\notin \{j_1,...,j_n\}$. Then $T$ admits a finite-dimensional, linear representation $(\f,\tilde U_{\f})$, where $\tilde U_{\f}$ is the matrix which is restriction of $\tilde U$ to $\f$.
\end{proposition}
\begin{proof} The condition  ``$u_{j_kl}=0$ whenever $k\in{1,...,n},l\notin \{j_1,...,j_n\}$" assures  that
\be
\f\circ T=A\f,
\ee
where $A=\tilde U_{\f}$. Namely, the time evolution of functions in $\f$ projected on any subspace that does not contain any of the functions in $\f$ is $0$.
\end{proof}
It is interesting that the finite section method can reveal nonlinear representations, too:
\begin{proposition}
Let $(M,\mu)$ be a measure space and $T:M\rar M$. Let $\f=(f_{j_1},...,f_{j_n})$ be a subset of a basis in a Hilbert space of observables $\cO =L^2(\mu)$. Let $u_{j_kl}=0$ whenever $k\in{1,...,n}$, $f_l\neq F(\f)$. Further assume that there are $l_k\in \bN,k=1,...,K$ such that $f_{l_k}=F_k(\f)$ where $F_k$ is nonlinear. Then $T$ admits a finite-dimensional, nonlinear representation $(\f,\F)$, where $\F$ is given by
\be
\F(\f)=A\tilde\f, 
\ee 
where $\f\subset\tilde \f$ and $A$ is an $n\times (n+K)$ matrix.
\end{proposition}
\begin{proof} The condition ``$u_{j_kl}=0$ whenever $k\in{1,...,n}$, $f_l\neq F(\f)$" assures that the time evolution of $\f$ under $T$ is a (nonlinear) function of $\f$ only, and 
$
\F(\f)=A\tilde\f, 
$
because $f_{j_k}\circ T=\sum_{m\in \{j_1,...,j_n,l_1,...,l_K\} } u_{j_km}f_m$.
\end{proof}

\section*{Extensions}
\subsection*{Koopman Operator and Control Systems}

The relationship between Koopman operator theory and control theory have been explored intensely over the last decade \cite{mauroyetal:2020}.
Control systems in discrete time are defined on the product space $M\times N$,
\bea
\m'&=&T(\m,\n) \nonumber \\
\n'&=&W(\n),
\label{eq:cont}
\eea
where $\m\in M,\n\in N$.  The system (\ref{eq:cont}) is a skew-product system \cite{Petersen:1983}. The $N$ system is the control system.  Physically, the assumption is that the $N$ system   is separated from the $M$ system, and it possesses its own internal dynamics described by $W$. A simple additive, linear structure for the representation is obtained if we can find $\f:M\rar \C$ and $\u:N\rar \C$ such that
\be
\f'=A\f+B\u
\ee
Note that $\u$ is physically the input to the system. This structure has been exploited in \cite{korda2018linear} to establish connection to Model Predictive Control, where $\f$ was chosen to be a simple library of functions on $M$ and $\u$ the vector of time-delayed inputs. Feedback relationship of the type
\be
\u=\H(\f)
\ee
couples the two systems.
 The linear representation can be related to a nonlinear finite dimensional representation by learning the associated conjugacy, as in \cite{folkestad2020extended}.

\subsection*{Static Koopman Operator}

We now consider a map $T$ between different spaces  $T:M\rar N$. The set $\cO_M$ of all complex functions $f:M\rar \bC$ is the space of observables on $M$, while the set $\cO_N$ of all complex functions $g:N\rar \bC$ is the space of observables on $N$. Both are  linear vector spaces over the field of complex numbers. A finite-dimensional representation $(\f,\g,\F)$ of $T$  is a set of functions $\f=(f_1,...,f_n),\g=(g_1,...,g_m)$ and a mapping $\F$ such that
\be
\g(T\m)=\F(\f(\m)),
\label{srep}
\ee
where $\F:\bC^n\rar\bC^m$ and $(n,m)$ is the dimension of the representation. We again have the notion of a {\it faithful} representation:
\begin{definition} A representation $(\f,\g,\F)$ of $T:M\rar M$ is called faithful provided $\f:M\rar\f(M),\g:N\rar\g(N)$ are injective.
\end{definition}
The previous notion of representation is recovered when $M=N$ and $\f=\g$.

Any  map $T:M\rar N$ defines the {\it pullback} operator $U:\cO_N\rar\cO_M$ by
\be
Ug(\m)=g(T\m).
\ee
This is the linear composition operator associated with $T:M\rar N$. The image ${\cal{I}}(U)$ of $U$ is the set of functions $f:\cO_M\rar \bC$ that are constant on level sets $L_{\n}$ of $T$:
\be
L_\n=\{\m\in M|\m\in T^{-1}(\n) \}.
\ee
The space ${\cal{I}}(U)$ is a linear subspace of $\cO_M$.

Provided $\cO_M$ is a separable Hilbert space, and ${\cal{I}}(U)$ is closed, we can also define the pushforward operator $P:\cO_M\rar\cO_N$ by
\be
P(f)=(\Pi f) \circ T,
\ee
where $\Pi:\cO_M\rar {\cal{I}}(U)$ is the orthogonal projection operator onto ${\cal{I}}(U)$. The singular value decomposition is valid for bounded operators between separable Hilbert spaces  and thus we have the following characterization of
${\cal{I}}(U)$:

\begin{proposition}  Let $\cO_N,\cO_M$ be separable Hilbert spaces, and ${\cal{I}}(U)$ is closed. Then the space ${\cal{I}}(U)$ is orthogonal to the subspace at singular value $0$. In addition,  $U$ is  the pseudoinverse of $P$.
\end{proposition}
\begin{proof} The kernel of $P$, consisting of functions orthogonal to ${\cal{I}}(U)$, is the subspace of  $\cO_M$ corresponding to singular value $0$ of $P$. 
We also have
\be
UPf=\Pi f.
\ee
proving that $U$ is the pseudoinverse of $P$. 
\end{proof}
With a little bit of topological and measure-theoretic infrastructure, we can characterize the projection operator $\Pi$
\begin{theorem} Let $M,N$ be two Radon spaces - separable metric spaces on which every probability measure is a Radon measure. Assume that $M$ is endowed with a Borel measure $\mu$, and $T$ is a measurable map. Let $N$ can be endowed with the measure $\nu (A)=\mu(T^{-1}(A)$. Then 
\be
\Pi(f)=\bE(f|{\cal B}_{T^{-1}}),
\label{eq:cond}
\ee
where $\bE(f|{\cal B}_{T^{-1}})$, the conditional expectation of $f$ with respect to the sigma algebra induced by  $T$, is the orthogonal projection of $f$ on ${\cal I}(U)$.
\end{theorem}
\begin{proof} First observe that $\Pi(f)$ as defined in (\ref{eq:cond}) is in  ${\cal I}(U)$ and defines a projection, since applying conditional expectation twice yields the same result as applying it once. We need to prove that $f-\bE(f|{\cal B}_{T^{-1}})$ is orthogonal to  ${\cal I}(U)$, i.e. 
\be
\int_Mh^c (f-\bE(f|{\cal B}_{T^{-1}}))d\mu=0,
\ee 
where $h\in {\cal I}(U)$. Any function in ${\cal I}(U)$ is constant on level sets of $T$, and by disintegration of measure theorem
\bea
&  &\int_Mh^c (f-\bE(f|{\cal B}_{T^{-1}}))d\mu \nonumber \\
& & =\int_M\int_{T^{-1}(\n)}h^c (f-\bE(f|{\cal B}_{T^{-1}}))d\mu_\n d\nu(\n) \nonumber \\
                                                                & &=\int_Mh^c \left(\int_{T^{-1}(\n)}(f-\bE(f|{\cal B}_{T^{-1}}))d\mu_\n\right)d\nu(\n) \nonumber \\
                                                                & & =0. \nonumber \\
\eea
since 
\be
\int_{T^{-1}(\n)}f d\mu_\n=\bE(f|{\cal B}_{T^{-1}}).
\ee
\end{proof}
Now we have
\begin{proposition} The set of functions $(\f,\g,\F):M\times N\times \bC^k\rar \bC^k\times \bC^l\times \bC^l$ is a finite-dimensional representation of $T$
iff
\be
U\g(\m)=\F(\f(\m)).
\ee
\end{proposition}
The representation is linear provided $\F$ is a finite dimensional  $k\times l$ matrix:
\be
U\g(\m)=A(\f(\m))
\ee
To get an approximation to a finite-dimensional linear representation, we may select basis $f_1,...,f_m,...)$ on $\cO_M$ and $g_1,...,g_n,...$ on $\cO_N$,
and construct the representation of $P$. We assume we have access to $N$ realizations of data pairs corresponding to $(\m_j,\n_j=T(\m_j)),j=1,...,N$. The data 
points are $\f=(f_1(m_j),...,f_m(m_j))^T$ and $\g=(g_1(n_j),...,g_n(n_j))^T$. We form matrices 
\bea
X&=&[\f(\m_1)...f(\m_N)],\nonumber \\
Y&=&[\g(\n_1)...g(\n_N)].
\eea
The solution to
\be
\min_A ||Y-AX||
\ee
is 
\be
A=YX^+
\ee
where $X^+$ is the pseudoinverse of $X$.

\section*{Neural Networks and the Koopman Operator}
The key to the DMD-type approximations to the Koopman operator are the predetermined basis functions. In the case of EDMD \cite{williamsetal:2015}, these are selected a-priori, and in the case of Hankel-DMD \cite{ArbabiandMezic:2017} they are generated using an initial choice of observables supplemented by time-delayed observables generated by dynamics. While EDMD suffers from curse of dimensionality, Hankel-DMD does not, as in any dimension the generated functions fill up an invariant subspace of the Koopman operator. However, there is no guarantee that there is a (linear or nonlinear) finite representation amongst the observables in either case. GLA  solves that problem by computing the spectrum and then computing the eigenfunctions by weighted time averages. As we have seen, eigenfunctions provide us with linear representations. The deep neural network formalism has been used to compute linear representations, where both the observables and the eigenmatrix $A$ are learned \cite{li2017extended,yeung2019learning,lusch2018deep,takeishi2017learning}. 

The neural network formulation for the solution of the representation eigenproblem (\ref{eq:repeig}) is
\be
\min_{(\theta,\psi)} \sum_{j=1}^m || \n_{\cO}(\m_{j+1},\theta)-\n_\F(\n_{\cO}(\m_{j},\theta),\psi)||
\ee
where $\n_{\cO}(\m_{j+1},\theta)$ is the neural network representing the observables, with parameters $\theta$, and 
$\n_\F(\n,\psi)$ is the neural network representing the eigenmap, with parameters $\psi$. The dimension of the vector $\n$ (and thus the dimension of $\F$) is a hyperparameter.

\section*{Relationship to Mori-Zwanzig Formalism}
The Mori-Zwanzig formalism describes evolution of a subset of observables $\f\in \cO$, where $\cO$ is a Hilbert space, using the Koopman operator and orthogonal projection $P$ on the subspace spanned by $\f=\{f_1,...,f_n\}$. Using $P$ and $Q=I-P$ we get
\be
U^t\f=(P+Q)U^t \f=\tilde U^t_{n}\f+QU^t\f,
\ee
where $\tilde U^t_{n}$ is the finite section matrix (\ref{comp}), and $Q_nU^t\f$ is the projection of the evolution of $\f$ on the space orthogonal to the span of $\f$ in $\cO$. It is immediately clear that, provided $(\f,\tilde U_{n}^t)$ is a linear representation, $QU^t\f$ is zero, and thus we get 
\be
U^t\f=\tilde U^t_{n}\f.
\ee
Assuming $U^t_n$ is diagonalizable, we get
\be
U^t\f=\Phi D^t\Phi^{-1}\f.
\ee
where $D^t$ is a diagonal matrix containing $\lambda_j^t, j=1,...,n$  on the diagonal, where $\lambda_j$ is an eigenvalue of $\tilde U^t_{n}$.

In discrete time, the evolution reads
\bea
U^{2}&=&(P+Q)U(PU +QU) \nonumber \\
&=&(PU)^2+PUQU+QUPU+(QU)^2\nonumber 
\eea
and, by induction
\bea
U^n&=&(PU)^n\nonumber \\
&+&\sum\limits_{\substack{k_1,...,k_n=0\\\sum_j k_j=n-1}}^n (PU^{k_1}QU^{k_2}...PU^{k_1}QU^{k_2})
\nonumber \\
&+&(QU)^n \nonumber \\
&=&\tilde U_n\nonumber \\
&+&\sum\limits_{\substack{k_1,...,k_n=0\\\sum_j k_j=n}}^{n-1} (PU^{k_1}QU^{k_2}...PU^{k_1}QU^{k_2})
\nonumber \\
&+&(QU)^n \nonumber
\label{discMZ}
\eea
The second term is often interpreted as the ``memory term" but in fact it contains a total of $n$ applications of $U$ just like the first and the last term. Rather, it describes the part of the evolution that depends on the evolution in both the span of $\f$ and its orthogonal complement. The following result clarifies the point, and identifies the situation in which the evolution in projected variables is Markovian:
\begin{proposition}If the evolution in the  orthogonal complement of $\f$ is dependent on $\f$ only, but is not $0$, i.e. $QU\f\neq 0$, then $T$ admits  a nonlinear representation $(\f,\F)$.
\end{proposition}
\begin{proof} Since $QU\f=\G(\f)$ for some $\G:\bC^n\rar \bC^n$ and $PU\f=\tilde U_n\f$ then
\be
U\f=PU\f+QU\f=\tilde U_n\f+\G(\f)=\F(\f).
\ee
where $\F:\bC^n\rar \bC^n$.
\end{proof}
The following example of the result in the above proposition also indicates the perils of modeling the $(QU)^n$ term in (\ref{discMZ}) as noise as it is commonly done in Mori-Zwanzig literature \cite{mori1965transport,zwanzig2001nonequilibrium,venturi2014convolutionless,lin2019data}.
\begin{example} Consider the irrational circle rotation $T:S^1\rar S^1$ defined by $z'=e^{i\omega}z$ where $z=e^{i\theta}$ and $\omega/2\pi$ is irrational. This  is an ergodic system on  $S^1$. We denote the complex inner product with respect to Haar measure $\mu$ on $S^1$ by $\left<\cdot,\cdot\right>$. Consider an analytic $L^2(\mu)$ observable $f:S^1\rar \bC$ that separates points on $S^1$, namely $f(z_1)=f(z_2)\Rightarrow z_1=z_2$. The Taylor expansion of $f$ is given by 
\be
f=\sum_{n=0}^\infty c_n z^n,
\label{tayexp}
\ee
where $c_n\in \C$. We have 
\be
Uf=f'=\sum_{n=0}^\infty c_n e^{in\omega}z^n.
\ee
Denoting the complex conjugate by $(\cdot)^c$, we define the complex scalar $\lambda$ by
\be
\lambda=\left<f',f\right>=\sum_{n=0}^\infty c_nc_n^c e^{in\omega}\left<z^n,(z^n)^c\right>=\sum_{n=0}^\infty |c_n|e^{in\omega}.
\ee
We have
\be
f'=\lambda f+QUf.
\ee
Note that
\be
QUf=\sum_{n=0}^\infty c_n (1-\lambda) e^{in\omega}z^n=(1-\lambda)Uf.
\ee
Since $f$ separates and $T$ is a bijection, $f(z_1)=f(z_2)\Rightarrow z_1=z_2 \Rightarrow Uf(z_1)=Uf(z_2)$,
and thus $QUf=(1-\lambda)Uf=G(f)$ where $G:f(S^1)\rar \bC$. Thus,
\be
f'=\lambda f+Gf.
\label{clos}
\ee
is Markovian, i.e. contains no memory terms, and is ``closed" i.e. the term $QUf$ should not be modeled as noise. The result can hold even of $f$ does not separate points on $S^1$. namely, it is easy to see that for any $n$, $f=z^n$ leads to $c_n=1,c_{j\neq n}=0, \lambda=e^{in\omega},G(f)=0$ and the evolution reads
\be
f'=e^{in\omega} f.
\ee
reflecting the fact that $f$ is an eigenfunction of $U$ with eigenvalue $e^{in\omega}$.
\end{example}
The main result (\ref{clos}) in the above example is also true for dynamical systems in Hilbert functional spaces in which the Koopman operator has point spectrum \cite{Mezic:2019}, and there is a set of separating {\it principal eigenfunctions} \cite{MohrandMezic:2014} of $U$ or $U^t$. Provided the set of observables $\f$ separates, we can use the algebraic property of eigenfunctions (\ref{alg}) then the evolution of $\f$ under $U$ reads
\be
\f'=A\f+\G(\f).
\ee
and there are no noise and memory terms.

From the foregoing analysis   it becomes clear that the noise and memory terms in the Mori-Zwanzig framework arise due to 1) the fact that the chosen representation $\f$ is not faithful and 2) the spectrum of the Koopman operator associated with a dynamical system has a continuous part.
\section*{Conclusions and Futures}

We provided a framework for learning of dynamical systems rooted in the concept of representations and Koopman operators. The interplay between the two leads to the full description of systems that can be represented linearly in a finite dimension, based on the properties of Koopman operator spectrum. As shown here, even nonlinear representations can be learned using the Koopman operator framework. 

The essential difference in the type of learning happening in our brains and the type of supervised learning dominant in machine learning is the notion of time. Time 
is also at the core of understanding causal relationships. Namely, without time only correlation between observables is possible. The Koopman operator theory remedies this by explicitly taking time into account and providing this dimension of learning with the explicit mathematical structure. Moreover, based on techniques such as Generalized Laplace Analysis, that naturally yield themselves to adjustments using streaming data, unsupervised learning, leading to generative models, is achievable, where future data is adapted organically into the learned structure.  The approach thus provides a suitable setting for unsupervised learning, and extends to ``static" problems that do not incorporate time. 
\section*{Acknowledgements} This work was supported by This work was supported by ONR contracts N00014-18-P-2004 and N00014-19-C-1053 and DARPA contract HR0011-18-9-0033.

\bibliography{KvN,MOTDyS,MOTDyS2}

\end{document}